\documentclass{amsart}

\usepackage{amssymb, graphicx, amsrefs}

\copyrightinfo{2014}{Rados{\l}aw K. Wojciechowski}

\newtheorem{theorem}{Theorem}[section]
\newtheorem{lemma}[theorem]{Lemma}
\newtheorem{proposition}[theorem]{Proposition}
\newtheorem{corollary}[theorem]{Corollary}

\theoremstyle{definition}
\newtheorem{definition}[theorem]{Definition}
\newtheorem{example}[theorem]{Example}

\theoremstyle{remark}
\newtheorem{remark}[theorem]{Remark}

\numberwithin{equation}{section}

\newcommand{\R}{{\mathbb R}}
\newcommand{\N}{{\mathbb N}}
\newcommand{\Lp}{{\Delta}}

\newcommand{\Deg}{{\mathrm{deg}}}

\newcommand{\supp}{{\mathrm {supp}\,}}
\newcommand{\pdt}{\frac{\partial}{\partial t}}

\newcommand{\bd}[1]{\partial \om_{#1}}
\newcommand{\inte}[1]{\mathring \om_{#1}}
\newcommand{\bbr}[1]{{\partial B({#1})}}

\newcommand{\ka}{{\kappa}}
\newcommand{\vka}{{\ow{\kappa}}}

\newcommand{\om}{{\Omega}}
\newcommand{\ep}{{\epsilon}}
\newcommand{\lm}{{\lambda}}
\newcommand{\h}{{{h}}}

\newcommand{\as}[1]{\left\langle #1\right\rangle}
\newcommand{\av}[1]{\left\vert #1\right\vert}
\newcommand{\aV}[1]{\left\Vert #1\right\Vert}
\newcommand{\ov}[1]{\overline{ #1}}
\newcommand{\ow}[1]{\widetilde{ #1}}
\newcommand{\oh}[1]{\widehat{ #1}}

\providecommand{\eat}[1]{}

\begin{document}

\title{The Feller Property for Graphs}
\author{Rados{\l}aw K. Wojciechowski}
\address{Graduate Center of the City University of New York, 365 Fifth Avenue, New York, NY, 10016.
}
\address{York College of the City University of New York, 94-20 Guy R. Brewer Blvd., Jamaica, NY 11451.}
\email{rwojciechowski@gc.cuny.edu}
\date{\today}
\thanks{The author gratefully acknowledges financial support from PSC-CUNY Awards, jointly funded by the Professional Staff Congress and the City University of New York, and the Collaboration Grant for Mathematicians, funded by the Simons Foundation.}

\begin{abstract}
The Feller property concerns the preservation of the space of functions vanishing at infinity by the semigroup generated by an operator.  We study this property in the case of the Laplacian on infinite graphs with arbitrary edge weights and vertex measures.  In particular, we give conditions for the Feller property involving curvature-type quantities for general graphs, characterize the property in the case of model graphs and give some comparison results to the model case.  
\end{abstract}

\maketitle

\section{Introduction}
Any semigroup which, among other properties, preserves the space of functions which vanish at infinity is called a Feller semigroup.  The continuity properties of such semigroups and the associated processes were studied early on by Feller \cite{Fel1, Fel2}.  For the semigroup generated by the Laplacian on a Riemannian manifold this property has been investigated extensively; we mention the works of Azencott \cite{Az}, Yau \cite{Y}, Dodziuk \cite{D}, Karp and Li \cite{KLi},  Hsu \cite{Hsu1}, and Davies \cite{Da}, amongst many others.  A Riemannian manifold for which the semigroup generated by the Laplacian is Feller is said to satisfy the \emph{Feller property} or the $C_0$-\emph{diffusion property} or is said to be simply \emph{Feller} for short.   This property can clearly be interpreted in terms of the heat flow and the Brownian motion. A survey of known result with many new conditions was given in the recent article of Pigola and Setti \cite{PS} which we follow throughout our presentation.  Our aim is to study this property for the semigroup generated by the Laplacian on a locally finite infinite graph with arbitrary edge weights and vertex measures. 

In general, it seems that there are two ways in which a space will satisfy the Feller property.  One way is via sufficient restrictions on the growth of the space which ensures that heat will not spread too far from its original starting point.  Thus, in the Riemannian setting, lower bounds on the Ricci curvature are usually used to imply the Feller property.  We mention here the result of Yau \cite{Y} which states that if the Ricci curvature is uniformly bounded from below, then the Riemannian manifold is Feller.  Dodziuk \cite{D} reproved this result using a maximum principle approach and we adapt this proof here to  show that if a curvature-type quantity is uniformly bounded from below, then the graph is Feller, see Theorem \ref{general_thm} in Section \ref{s:graph}.  This holds, in particular, if the Laplacian is a bounded operator, see Corollary \ref{general corollary}.  An important distinction should be highlighted in that, unlike similar criteria for stochastic completeness in the graph setting \cite{We, Hua}, we need to assume that the curvature-type quantity is uniformly controlled in \emph{all} directions. 

The second, and much less emphasized, way in which a space can be Feller is via rapid growth which forces heat to infinity rapidly where it dissipates.  This can already be seen in the work of Azencott \cite{Az} which characterizes the Feller property in the case of spherically symmetric or model manifolds and implies that if a model manifold is transient, then it is Feller.  The analogue to model manifolds in the graph setting was recently developed and studied in \cite{KLW}.  We prove a counterpart to Azencott's characterization following the approach given by Pigola and Setti though our result looks slightly different due to the presence of an arbitrary vertex measure, see Theorem \ref{model theorem} in Section \ref{s:graph}.  In particular, all transient model graphs are Feller.

The result on models allows us to give examples of graphs which are not Feller.  These are spherically symmetric graphs for which the boundary of balls does not grow too rapidly and for which the measure decays.  The decay of the measure accelerates the process sufficiently so that the bulk of it does not remain within a confined region, but the lack of boundary growth ensures that the heat does not dissipate at infinity either.  As such,  heat congregates at infinity and the graph is not Feller.  In general, if the measure does not decay in at least one direction, then the graph is Feller, irrespective of the edge weights, see Proposition \ref{measure proposition}.  In the case of models, this result can be strengthened to say that all model graphs of infinite measure are Feller, see Corollary \ref{model corollary}.

We also prove a comparison theorem which states that if a general graph has stronger curvature growth than a model graph which is Feller, then the graph is Feller and if a general graph has weaker curvature growth than a model graph which is not Feller, then it is not Feller, see Theorem \ref{comparison_theorem} in Section \ref{s:graph}.  We note that the comparisons are the opposite of what would be expected given similar comparison theorems for stochastic completeness and for the bottom of the spectrum \cite{KLW}.  

In this light, it is surprising that no general growth condition on our curvature-type notion implies the Feller property.  In fact, we can create examples of graphs of arbitrarily large curvature growth in all directions which are not Feller, see Example \ref{anti_example}.  However, by modifying the notion of curvature slightly, such a general result is possible, see Theorem \ref{general_thm2}.  A similar result can also be given for a stochastically complete graph to be non-Feller, see Theorem \ref{general_thm3}.  In Subsection \ref{sub_ex_stab} we give examples of graphs where these twisted curvature-type criteria can be applied in the study of the stability of the Feller property. Finally, we would like to mention a recent extension of the Feller property called the \emph{uniform strong Feller} property and the connection of this property to a strong form of transience developed in a forthcoming paper \cite{KLSW}.

This paper is structured as follows: in Section \ref{s:setting} we introduce the general setting of weighted graphs, Laplacians, and the heat semigroup and prove some maximum principles which will be used throughout.  Section \ref{s:Feller} introduces the Feller property, gives an alternative perspective via an elliptic characterization and some related criteria which hold in general.  Finally, in Section \ref{s:graph}, we prove the various criteria for the (failure of) Feller property on graphs mentioned above and give examples to illustrate our results.

\subsection*{Acknowledgments} The author expresses his gratitude to J{\'o}zef Dodziuk for his insights, support and generosity over the years.  Furthermore, the author would like to thank the University of Jena for facilitating several visits and Matthias Keller, Daniel Lenz and Marcel Schmidt for numerous helpful discussions.

\section{Setting}\label{s:setting}
In this section we introduce the setting of weighted graphs, discrete Laplacians, the associated semigroup and then prove some maximum principles which will be used throughout.

\subsection{Weighted graphs}
We work in the context of general weighted graphs as established in \cite{KL} except that we assume that all graphs are locally finite and have no killing term.  For some results, the local finiteness is not an essential requirement as we point out where applicable.

By \emph{weighted graph} $G$ we mean a triple $G=(X, b, m)$ where $X$ is a countably infinite set whose elements are called \emph{vertices}, $b$ is the \emph{edge weight}  which gives the graph structure and $m$ is a positive \emph{vertex measure}.

More specifically,
$b: X \times X \to [0, \infty)$
satisfies $b(x,y) = b(y,x)$, $b(x,x)=0$, and
$ \av{ \{ y \ | \ b(x,y) > 0 \} } < \infty \textup{ for all } x \in X.$
This last assumption is what is referred to as \emph{local finiteness}.  More generally, as in \cite{KL}, we need only assume that $\sum_{y \in X} b(x,y) < \infty$ for all $x \in X$ to get a reasonable setup.  Pairs of vertices $x$ and $y$ such that $b(x,y)>0$ are said to be \emph{connected} by an \emph{edge} with weight $b(x,y)$.  We write $x \sim y$ in this case.  Finally, $m: X \to (0, \infty)$ is called the \emph{vertex measure} and is extended to all subsets of $X$ by additivity.

Two vertices $x$ and $y$ are said to be \emph{connected} if there exists a sequence of vertices $(x_i)_{i=0}^n$ such that $x_0 = x$, $x_n = y$ and $x_i \sim x_{i+1}$ for $i=0, 1, \ldots, n-1$.  Such a sequence is called a \emph{path} connecting $x$ and $y$.  A graph $G$ is said to be \emph{connected} if every pair of vertices is connected.  
\[ \textup{We assume throughout that $G$ is connected.} \]

We will work with the standard combinatorial metric on graphs which is given by counting the number of edges in the shortest path connecting two vertices.  We denote this metric by $d$.  That is, denoting by $\Gamma_{x,y}$ the set of all paths connecting $x$ and $y$ and, for $\gamma = (x_i)_{i=0}^n \in \Gamma_{x,y}$, letting $l(\gamma) = n$ be the \emph{length} of a path $\gamma$, we have
\[ d(x,y) = \inf_{\gamma \in \Gamma_{x,y}} l(\gamma). \]


\subsection{Function spaces and Laplacians}
We denote by $C(X)$ the space of all real-valued functions on $X$, that is, $C(X) = \{ f: X \to \R \}.$  Two important subspaces are the space of finitely supported functions and the closure of this space with respect to the sup norm, which consists of functions vanishing at infinity:
\begin{align*}
C_c(X) &= \{ f \in C(X) \ | \ \av{ \supp{f}} < \infty \} \\
C_0(X) &= \{ f \in C(X) \ | \ f(x_n) \to 0 \textup{ as } x_n \to \infty \} = \ov{C_c(X)}^{\aV{\cdot}_\infty} .
\end{align*}
Here, $x_n \to \infty$ means that the sequence eventually leaves every finite set, never to return, and $\aV{f}_\infty= \sup_{x \in X} \av{f(x)}.$  We denote the space of all bounded functions by $\ell^\infty(X)$.

To introduce the Laplacian, we need to specify the Hilbert space on which this operator will act.  We denote the square-summable functions with respect to $m$ by $\ell^2(X,m)$:
\[ \ell^2(X,m) = \{ f \in C(X) \ | \ \sum_{x \in X} f(x)^2 m(x) < \infty \} \]
with inner product given by
\[ \as{f, g} = \sum_{x \in X} f(x) g(x) m(x) \]
and $\aV{\cdot}$ the associated norm.
The \emph{formal Laplacian} $\ow{\Lp}$ acts on the space $C(X)$ by
\[ \ow{\Lp} f(x) = \frac{1}{m(x)} \sum_{y \in X} b(x,y) (f(x) - f(y) ). \]
Restricting $\ow{\Lp}$ to $C_c(X)$ gives a symmetric operator $\Lp_0 = \ow{\Lp}|_{C_c(X)}$ and we denote the Friedrichs extension of $\Lp_0$ by $\Lp$.  This self-adjoint extension comes from the closure of the quadratic form $\ow{Q}(f,g) = \frac{1}{2} \sum_{x,y \in X} b(x,y) (f(x)-f(y))(g(x)-g(y))$ acting on $C_c(X)$.  Note that, in general, there may be many self-adjoint extensions of $\Lp_0$.  For more details on this issue see \cite{HKMW} and the references therein.

We also point out that letting
\[ \Deg(x) = \frac{1}{m(x)}\sum_{y \in X} b(x,y) \]
denote the \emph{weighted degree} of a vertex $x$, it follows by \cite[Theorem 11]{KL2} that $\Lp$ is bounded as an operator on $\ell^2(V,m)$ if and only if $\deg \in \ell^\infty(X)$.


\subsection{Heat semigroup and heat kernel}
We will consider the \emph{heat semigroup} $P_t = e^{-t\Lp}$ which can be extended to act on $\ell^\infty(X)$ and the associated \emph{heat kernel} $p_t(x,y)$, for $t\geq0$ and $x,y \in X$, given by
\[ P_t f(x) =  e^{-t\Lp} f(x) = \sum_{y \in X} p_t(x,y) f(y) m(y) \]
for any $f \in \ell^\infty(X)$.  Given any bounded function $u_0$, it follows that $u(x,t) = P_t u_0(x)$ is the minimal bounded solution to the \emph{heat equation} with initial condition $u_0$:
 \[
\left\{
\begin{array}{rll}
\left( \ow{\Lp} + \pdt \right) u(x,t)  &= 0 &    x \in X, \ t\geq0 \\
 u(x,0)    &= u_0(x) &      x \in X.
\end{array}
\right.
\]
Minimality here means that if $u_0 \geq 0$, then $u(x,t)=P_t u_0(x)$ is the smallest non-negative solution to the heat equation with initial condition $u_0$.

A graph is called \emph{stochastically complete} if 
\[ \sum_{y \in X} p_t(x,y) m(y) = 1 \]
for all $x \in X$ and all $t\geq0$.  This is equivalent to the uniqueness of bounded solutions to the heat equation above, see \cite[Theorem 1]{KL}.

\eat{
There is a way of constructing the heat kernel corresponding to the minimal Laplacian via exhaustion sequences which goes back to the work of Dodziuk on manifolds \cite{D}, see \cite{W08, W09, We, KL} for details in the case of graphs.  Namely, take a sequence $(\om_n)_{n=0}^\infty$ of finite, connected subgraphs of $G$ such that $\om_n \subseteq \om_{n+1}$ and $\cup_{n=0}^\infty \om_n = X$.  Any such sequence is called an \emph{exhaustion sequence} of $G$.  Now, for each $n$, let
\[ \bd{n} = \{ x \in \om_n \ | \ \exists y \sim x \textup{ such that } y \not \in \om_n \} \]
denote the \emph{vertex boundary} of $\om_n$ and let $\inte{n} = \om_n \setminus \bd{n}$
denote the \emph{interior} of $\om_n$.

Furthermore, consider the \emph{Dirichlet Laplacian} $\Lp_n$ on $\om_n$ which is defined by
\[ \Lp_n f(x) = \left\{
\begin{array}{cl}
\Lp f(x) & \textup{ for } x \in \inte{n} \\
0 & \textup{ for } x \in \bd{n}
\end{array}
\right. \]
and the corresponding \emph{Dirichlet heat kernels} $p_t^n(x,y)$ given by
\[ P_t^n f(x) = e^{-t \Lp_n} f(x) = \sum_{y \in \inte{n}} p_t^n(x,y) f(y). \]
It follows by standard maximum principle arguments, see Proposition \ref{p maximum principle} directly below, that $p_t^{n+1}(x,y) \geq p_t^n(x,y)$ and $\lim_{n \to \infty} p_t^n(x,y) = p_t(x,y).$
\eat{The technique of exhausting and taking the limit and the fact that $P_t^n f(x) \to P_t f(x)$ for $f \in \ell^\infty(X)$ will be used repeatedly below.}
}

\subsection{Maximum principles}
We present here some basic maximum principles in both elliptic and parabolic forms.  These are certainly well-known in our setting, see, for example, \cite{D2, DM, Hua4, KL, KLW, We, W08, W09}.  
As such, we omit the proofs.

For a finite, connected subgraph $\om$ of $G$, let
\[ \bd{} = \{ x \in \om \ | \ \exists y \sim x \textup{ such that } y \not \in \om \} \]
denote the \emph{vertex boundary} of $\om$ and let $\inte{} = \om \setminus \bd{}$
denote the \emph{interior} of $\om$.

\begin{proposition}\label{e maximum principle}
Let $\om$ and  $\om_1$ be finite, connected subgraphs of $G$ such that $\om \subseteq \inte{1}$.  If $\lm<0$ and $v \in C(\om_1)$ satisfies
\[ \left\{
\begin{array}{ll}
\Lp v = \lm v & \textup{on } \inte{1}\setminus \om \\
v = 1 & \textup{on } \bd{},
\end{array} \right. \]
then $0<v<1$ on $\inte{1} \setminus \om.$
\end{proposition}
\eat{\begin{proof}
Suppose that there exists $x \in \inte{1} \setminus \om$ such that $v(x) \leq 0$.  We may assume that $x$ is a minimum for $v$ on $\inte{1} \setminus \om$.  It then follows that $\Lp v(x) \leq 0$.  On the other hand, $\Lp v(x) = \lm v(x) \geq0$ so that $\Lp v(x) =0$.    It then follows that $v(x) = v(y)$ for all neighbors $y \sim x$.  Repeating this argument gives a contradiction to the fact that $v=1$ on $\bd{}$.

If there exists $x \in \inte{1} \setminus \om$ such that $v(x)\geq1$, then we may assume that $x$ is a maximum for $v$ on $\inte{1} \setminus \om$.  Hence, $\Lp v(x) \geq 0$ while $\Lp v(x) = \lm v(x) < 0$ which gives a contradiction.
\end{proof}
}

\begin{proposition}\label{p maximum principle}
Let $\om$ be a finite, connected subgraph of $G$.  If $u:\om \times [0,T] \to \R$ satisfies
\[ \left( \Lp + \pdt \right) u(x,t) \leq 0 \textup{ on } \inte{} \times [0,T], \]
then 
\[ \max_{\om \times [0,T]} u(x,t) = \max_{\substack{ \om \times \{0\} \cup \\ \bd{} \times [0,T]}} u(x,t). \] 
\end{proposition}
\eat{\begin{proof}
If $(x,t) \in \inte{}\times (0,T]$ is a maximum for $u$, then
\[ \Lp u(x,t) \geq 0 \textup{ and } \pdt u(x,t) \geq 0 \]
implying that $\Lp u(x,t)=0$.  Then $u(x,t) = u(y,t)$ for all $y \sim x$.  Repeating the argument and using the connectivity of $\om$ then implies that $u(\cdot, t)$ is constant.
\end{proof}
}


\section{The Feller Property}\label{s:Feller}
In this section we first define then give an elliptic characterization and some Khas{\cprime}minski{\u\i}-type criteria for the Feller property.  The proofs are formally the same as in the manifold setting in \cite{Az, PS}.  Thus, we omit the proofs below and rather point out the corresponding places in \cite{Az, PS}.  

\subsection{Definition}
We start with the definition of the Feller property and a simplification which will be used later. 

\begin{definition}
A weighted graph $G$ is said to be \emph{Feller} if
\[ P_t: C_0(X) \to C_0(X) \textup{ for all } t \geq 0. \]
\end{definition}

That is, $G$ is Feller if the semigroup arising from the Laplacian preserves the functions which vanish at infinity for every fixed time $t$. A simplification of this states that it suffices to check that non-negative, finitely supported functions are mapped to functions vanishing at infinity, compare \cite[Lemma 1.2]{PS}.

\begin{proposition}\label{finite support}
If $P_t u_0 \in C_0(X)$ for all $u_0 \geq 0$ with $u_0 \in C_c(X)$, then $G$ is Feller.
\end{proposition}

\eat{
\begin{proof}
Split an arbitrary $u \in C_0(X)$ into positive and negative parts, then use the linearity of $P_t$ and the fact that $P_t$ is continuous on $C_c(X)$ with respect to $\aV{\cdot}_\infty$.

\end{proof}
}


\subsection{An elliptic reformulation}
We now present a reformulation of the Feller property involving functions which are $\lambda$-\emph{harmonic} outside of a finite set, that is, satisfying $\ow{\Lp} f = \lambda f$ for some constant $\lambda$ outside of a finite set.  The original result is by Azencott \cite{Az}; a proof is also given in Pigola and Setti \cite[Theorem 2.2]{PS}.

\begin{theorem} \label{Az theorem}
The following statements are equivalent:
\begin{itemize}
\item[(1)] $G$ is Feller.
\item[(2)]  For some (any) $\om \subset X$ finite, for some (any) $\lambda<0$, $h: X \setminus \om \to (0, \infty)$ the minimal solution to
\begin{equation}\label{minimal solution}
\left \{
\begin{array}{ll}
\ow{\Lp} h = \lambda h & \textup{on } X \setminus \om \\
h =1 & \textup{on } \partial \om \\
h>0 & \textup{on } X \setminus \om
\end{array} \right.
\end{equation}
is in $C_0(X)$.
\end{itemize}
\end{theorem}

\begin{definition}
We call $h$ the \emph{minimal, positive solution to the} $\lambda$\emph{-harmonic exterior boundary problem} or just the \emph{minimal solution to} (\ref{minimal solution}).
\end{definition}

The minimal solution $h$ is constructed by an exhaustion procedure as follows. 
Take a sequence $(\om_n)_{n=0}^\infty$ of finite, connected subgraphs of $G$ such that $\om_n \subseteq \om_{n+1}$ and $\cup_{n=0}^\infty \om_n = X$.  
Let $h_n$ be functions satisfying $\Lp h_n = \lambda h_n$ on $\inte{n}\setminus \om$, 
$h_n =1$ on $\partial \om$ and $h_n = 0$ on $\bd{n}$.  By Proposition \ref{e maximum principle}, it follows that $0 < h_n < 1$ on $\inte{n} \setminus \om$ and $h_n \leq h_{n+1}$.   Finally, it can be shown that $h = \lim_{n \to \infty} h_n$ has the desired properties.

\eat{The proof of the theorem is formally the same as in \cite{PS}.  We sketch the details for the convenience of the reader.  We start by showing that $h$ can be constructed by an exhaustion sequence procedure as discussed for the heat kernel above.

\begin{proposition} \label{fundamental convergence prop}
Let $(\om_n)_{n=1}^\infty$ be any exhaustion sequence of $G$ such that $\om \subseteq \inte{1}.$  For each $n$, let $h_n$ satisfy
\[
\left \{
\begin{array}{ll}
\Lp h_n = \lambda h_n & \textup{on } \inte{n}\setminus \om \\
h_n =1 & \textup{on } \partial \om \\
h_n = 0 & \textup{on } \bd{n}.
\end{array}
\right.
\]
Then, the minimal, positive solution to (\ref{minimal solution}) satisfies
$h = \lim_{n \to \infty} h_n.$
\end{proposition}
\begin{proof}
By Proposition \ref{e maximum principle}, it follows that $0 < h_n < 1$ on $\inte{n} \setminus \om$ and a similar argument as given in the proof implies that $h_n \leq h_{n+1}$.  Hence, the convergence of the sequence $(h_n(x))_n$ for each $x$ follows.  Using the dominated convergence theorem, it is not hard to see that the limiting function is $\lambda$-harmonic.  Finally, using the maximum principle again, the limiting function must be positive and less than or equal to any other solution to (\ref{minimal solution}) by comparing with $h_n$.
\end{proof}
}

\begin{remark}
Even in the Feller case, there may be bounded solutions to (\ref{minimal solution}) which are not in $C_0(X)$, see Remark \ref{uniqueness remark}.  However, in the stochastically complete case, $h$ is the unique bounded solution to (\ref{minimal solution}), see Corollary \ref{uniqueness corollary} below.
\end{remark}

\eat{
\begin{proof}[Proof of Theorem \ref{Az theorem}] \

\noindent $(1) \Longrightarrow (2)$:  \quad  Let $u_0 \geq 0, u_0 \in C_c(X)$ and $u(x,t) = P_t u_0(x)$ which is in  $C_0(X)$ by assumption. Consider the function
\[ v(x) = \int_0^\infty u(x,t) e^{\lambda t} dt. \]
As $u$ is bounded, it follows that $v<\infty.$  Since the semigroup is positivity improving \cite[Theorem 7.3]{HKLW}, $v$ is positive and $v \in C_0(X)$ by the dominated convergence theorem.

Furthermore, using integration by parts, it follows that
\[ \ow{\Lp} v(x) = u_0(x) + \lambda v(x) \geq \lambda v(x). \]
Letting $C = \min_{x \in \partial \om} v(x)$ and $ w(x) = v(x)/C$, we get that $w>0, w \geq 1$ on $\partial \om, \Lp w \geq \lambda w$ and $w \in C_0(X)$.

Using a maximum principle argument as in Proposition \ref{e maximum principle}, it follows that $w > h_n$ on $\inte{n} \setminus \om$ for all $n$ and taking the limit we get that $w \geq h$.  Therefore, $h \in C_0(X)$.

\bigskip
\noindent $(2) \Longrightarrow (1)$: \quad  Let $u_0 \geq 0$ with $\supp u_0 = D$ where $D$ is finite and let $u(x,t) = P_t u_0(x)$.  By Proposition \ref{finite support}, it suffices to show that $u(\cdot, t) \in C_0(X)$ for evert $t\geq0$.  Let $h$ be the minimal solution to (\ref{minimal solution}) for some $\om$ finite and $\lambda<0$.   Choose an exhaustion $(\om_n)_{n=1}^\infty$ such that $D \cup \om \subseteq \inte{1}.$

Consider $u_n(x,t) = P_t^n u_0(x)$ where $P_t^n$ is the Dirichlet heat semigroup on $\om_n$ so that $u_n(x,t)$ is an increasing sequence such that $u_n(x,t) \to u(x,t) $ as $n \to \infty$.  Now, let $C>0$ be such that $Ch(x) \geq u(x,t)$ on $\bd{1} \times [0,T]$ and compare 
\[ u_n(x,t) \textup{ and }  w(x,t) = Ch(x) e^{-\lambda t}\]
on $\om_n \setminus \inte{1} \times [0,T].$

An easy calculation gives that $( \ow{\Lp} + \pdt ) w = 0.$
Then, on $\om_n \setminus \inte{1} \times \{0\}, u_n(x,0) = u_0(x) = 0$ while $w>0$.  Now, $\partial( \om_n \setminus \inte{1}) \subseteq \bd{n} \cup \bd{1}$ and on $\bd{n} \times [0,T], u_n(x,t) = 0$ by Dirichlet boundary conditions, while on $\bd{1} \times [0,T]$, $u_n(x,t) \leq u(x,t) \leq w(x,t)$ by the choice of the constant $C$.  Therefore, by the parabolic maximum principle, Proposition \ref{p maximum principle}, it follows that
\[ u_n(x,t) \leq w(x,t) \textup{ on } \om_n \setminus \inte{1} \times [0, T]. \]
Taking the limit gives that $u(x,t) \leq w(x,t)$ on $X \setminus \inte{1} \times [0,T]$.  Since $h \in C_0(X)$ by assumption, it follows that $w$ and, therefore, $u$ are in $C_0(X)$ for every $t$.
\end{proof}
}

\subsection{Khas{\cprime}minski{\u\i}-type criteria}

We also present some Khas{\cprime}minski{\u\i}-type criteria for the Feller property.  To compare, see similar tests for stochastic completeness and recurrence in \cite{Gri, Kh}. \eat{ The proofs are analogous to those found in \cite{PS}.}
These tests will be used to prove some general criteria for the Feller property and comparison theorems to the model case below.
The first result is an analogue to \cite[Proposition 5.1]{PS}.

\begin{theorem}\label{subharmonic}
If there exists a positive function $v$ such that for some $\lambda<0$ and some $\om \subset X$ finite
\[ \left\{
\begin{array}{ll}
\ow{\Lp} v \geq \lambda v &\textup{on } X \setminus \om \\
v \geq 1 & \textup{on } \partial \om,
\end{array}
\right. \]
then $v \geq h$ where $h$ is the minimal solution to (\ref{minimal solution}).  In particular, if $v \in C_0(X)$, then $G$ is Feller.
\end{theorem}

\eat{
\begin{proof}
Let $(\om_n)_{n=1}^\infty$ be an exhaustion sequence such that $\om \subset \inte{1}$.  Let $h_n$ satisfy
\[ \left\{
\begin{array}{ll}
\Lp h_n = \lambda h_n &\textup{on } \inte{n} \setminus \om \\
h_n = 1 & \textup{on } \partial \om \\
h_n = 0 & \textup{on } \bd{n}
\end{array}
\right. \]
so that $h_n \to h$.

Consider $w_n = v-h_n$.  Then
\[ \left\{
\begin{array}{ll}
\Lp w_n \geq \lambda w_n &\textup{on } \inte{n} \setminus \om \\
w_n \geq 0 & \textup{on } \partial \om \\
w_n > 0 & \textup{on } \bd{n}
\end{array}
\right. \]
from which it can easily be derived that $w_n > 0$ on $\inte{n} \setminus \om$.  By taking the limit, we obtain that $v \geq h$.
\end{proof}
}

Note that for the following two statements we need the additional assumptions that $G$ is stochastically complete and that $v$ is bounded.  The first result is an analogue to \cite[Theorem 5.3]{PS}, the second to \cite[Corollary 5.4]{PS}.

\begin{theorem} \label{superharmonic}
If $G$ is stochastically complete and there exists a positive, bounded function $v$ such that for some $\lm <0$ and some $\om \subset X$ finite
\[ \left\{
\begin{array}{ll}
\ow{\Lp} v \leq \lambda v &\textup{on } X \setminus \om \\
v \leq 1 & \textup{on } \partial \om,
\end{array}
\right. \]
then $h \geq v$ where $h$ is the minimal solution to (\ref{minimal solution}).  In particular, if $v \not \in C_0(X)$, then $G$ is not Feller.
\end{theorem}

\eat{
\begin{proof}
Let $w = v-h$ so that $\ow{\Lp} w \leq \lm w$ on $X \setminus \om$ and extend $w$ to $\om$ by 0.

Let $w_+(x) = \max \{ w(x), 0 \}$ and note that $\Lp w_+ \leq \lm w_+$ on $X$.  Therefore, $w_+$ is a bounded, non-negative, $\lambda$-subharmonic function.  As $G$ is stochastically complete, it follows that $w_+ \equiv 0$ \cite[Theorem 1]{KL}.  Hence, $v \leq  h$.
\end{proof}
}

By combining the previous two results, we immediately get the following corollary.  Note that stochastic completeness is necessary here, see Remark \ref{uniqueness remark} below.
\begin{corollary} \label{uniqueness corollary}
If $G$ is stochastically complete, then for every $\om \subset X$ finite and every $\lm<0$, $h$, the minimal positive solution to the $\lm$-harmonic exterior elliptic boundary problem is the unique bounded solution to (\ref{minimal solution}).
\eat{  That is, $h$ is the unique bounded function which satisfies
\[ \left \{
\begin{array}{ll}
\ow{\Lp} h = \lm h & \textup{on } X \setminus \om \\
h =1 & \textup{on } \partial \om \\
h>0 & \textup{on } X \setminus \om.
\end{array} \right.\]
}
\end{corollary}

\section{Graph Criteria}\label{s:graph}
In this section we prove several general curvature-type criteria for the Feller property on graphs, fully characterize the Feller property in the weakly spherically symmetric case and then prove some comparison theorems.  We also illustrate our results with examples and make some remarks concerning the stability of the Feller property.

\subsection{General criteria}
We first prove an analogue to a theorem of Yau which states that a Riemannian manifold with Ricci curvature bounded below satisfies the Feller property \cite{Y}.  This analogue, in particular, implies that if $\Lp$ is bounded, then the graph is Feller.  We follow the maximum principle approach used by Dodziuk in \cite{D} whose analogue for graphs has been developed in \cite{DM, D2, We}.

There are two equivalent ways of stating this criterion: one involves the Laplacian of a radial function based on the metric and one involves the difference of curvature-type quantities which are defined next and will play a substantial role in subsequent considerations. 

\begin{definition} 
Let $x_0 \in X$ and let $\rho(x) = d(x,x_0)$ where $d$ denotes the combinatorial graph metric.  If $x \in S_r(x_0) = \{ y \ | \ \rho(y)=r \}$, then we call
\[ \ka_\pm(x) = \frac{1}{m(x)} \sum_{y \in S_{r \pm 1}(x_0)} b(x,y) \] 
the \emph{outer} and \emph{inner} curvatures, respectively.
\end{definition}

\begin{theorem}\label{general_thm}
If $\ka_+(x) - \ka_-(x) \leq K $ for all $x \in X$ and all $x_0 \in X$ and some $K>0$, then $G$ is Feller.
\end{theorem}

\begin{remark}
\begin{itemize}
\item[(i)]  It is easy to see that the assumption of the theorem is equivalent to $\ow{\Lp} \rho(x) \geq -K$ for all $x \in X$ and all  $x_0 \in X.$
This is the condition implied by a uniform lower bound on the Ricci curvature in the Riemannian setting \cite[Lemma 2.3]{D}.

\item[(ii)]  The assumption of Theorem \ref{general_thm} also implies the stochastic completeness of the graph \cite[Theorem 4.15]{We}.  However, note that stochastic completeness only requires that this assumption hold for \emph{some} $x_0 \in X$.  For the Feller property, the assumption that the curvature bound holds for all $x_0$ is crucial, see Example \ref{model example}.  The result on stochastic completeness for graphs was later improved to allow some growth of $\ka_+ - \ka_-$.  More specifically, Huang showed that if $f>0$ is increasing, differentiable and satisfies
\[ \int^\infty \frac{1}{f(r)}dr=\infty, \]
then $\ka_+(x) - \ka_-(x) \leq f(\rho(x))$ for all $x$ and some $x_0$ implies stochastic completeness \cite[Theorem 5.4]{Hua}. 

In the case of manifolds, Yau's result was also improved by Hsu \cite{Hsu1} to allow some growth using probabilistic techniques.  The analogue of Hsu's result would say that if $f>0$ is increasing, differentiable and satisfies
\[ \int^\infty \frac{1}{\sqrt{f(r)}} dr = \infty, \]
then $\ka_+(x) - \ka_-(x) \leq f(\rho(x))$ for all $x$ and all $x_0$ should imply the Feller property.  This result would then be sharp as can be seen by Example \ref{model example}.

\eat{\item[(iii)]  The theorem could be extended to the non-locally finite case by replacing $d$ with any metric such that distance balls with respect to the metric are finite.  This assumption is equivalent to geodesic completeness, see \cite{HKMW}.}

\item[(iii)]  It is interesting to note that an assumption such as $\ka_+(x) -\ka_-(x) \geq K$ for all $x$ and all $x_0$ does not imply the Feller property.  In fact, we can construct examples of graphs with arbitrarily large curvature growth in all directions which are not Feller, see Example \ref{anti_example}.  This contrasts with Riemannian manifolds for which all Cartan-Hadamard manifolds are Feller, see \cite[Corollary 7.2]{PS}.  However, by modifying the definition of curvature growth slightly, such a result is possible, see Theorem \ref{general_thm2}.

\end{itemize}
\end{remark}

\begin{proof}
We follow the proof given in \cite[Theorem 4.3]{D}.

Let $u_0 \geq 0$ and suppose that $\supp u_0 = \om$ is finite.  Let $u(x,t) = P_t u_0(x)$  with $C_0(R) = \max_{x \in B_R(x_0) } u_0(x)$ and $C = \sup_{(x,t) \in X \times [0,T]} u(x,t).$  By Proposition \ref{finite support} it suffices to show that $u(\cdot, t) \in C_0(X)$.

Consider the function
\[ w(x,t) = u(x,t) - C_0(R) - \frac{C}{R} (Kt + \rho(x)) \]
on $B_R(x_0) \times [0,T]$ where $\rho(x) = d(x,x_0)$ and $B_R(x_0) = \{ x \ | \ \rho(x) \leq R \}$.  An easy calculation, using that $\ow{\Lp} \rho(x) + K \geq 0$ as mentioned in part (i) of the remark above, gives that
\[ \left( \ow{\Lp} + \pdt \right) w(x,t) \leq 0.  \]
Furthermore, $w(x,t) \leq 0$ on both $B_R(x_0) \times \{0\}$ and $\partial B_R(x_0) \times [0, T]$.  Therefore, using the maximum principle in parabolic form, Proposition \ref{p maximum principle}, gives that $w(x,t) \leq 0$ on $B_R(x_0) \times [0,T]$ 
so that
\[ u(x,t) \leq C_0(R) + \frac{C}{R} (Kt + \rho(x)). \]

Let $\epsilon>0$ and
suppose that $x_0$ is such that $d(x_0, \om) >R > \frac{C Kt}{\epsilon}$.  Then, $C_0(R) = 0$ on $B_R(x_0)$ and $\rho(x_0) = 0$.  Hence, we get that
\[ u(x_0,t) \leq \frac{C Kt}{R} < \epsilon. \]
Therefore, $u(x,t)$ is arbitrarily small outside of $B_R(\om) = \{x \ | \ d(x,\om) \leq R\}$ which is a finite set by local finiteness.  Hence, $u(\cdot,t) \in C_0(X)$ for all $t$.
\end{proof}

\begin{corollary}\label{general corollary}
If $\Lp$ is a bounded operator on $\ell^2(X,m)$, then $G$ is Feller.
\end{corollary}
\begin{proof}
As previously mentioned, $\Lp$ is bounded if and only if $\deg(x) = \frac{1}{m(x)} \sum_y b(x,y)$ is a bounded function on $X$ \cite[Theorem 11]{KL2}.  This is equivalent to $\ka_+(x)-\ka_-(x) \leq \deg(x)$ being bounded for all choices of $x_0$. 
\end{proof}

We also have the following criterion which involves the vertex measure of the graph only.  The local finiteness assumption is not necessary for this result to hold.
\begin{proposition}\label{measure proposition}
If $\inf_{x\in X} m(x)>0$, then $G$ is Feller.
\end{proposition}
\begin{proof}
This follows directly as $\ell^2(X,m) \subseteq C_0(X)$ in this case and $P_t(C_c(X)) \subseteq \ell^2(X,m)$ by the spectral theorem.
\end{proof}
\eat{Let $u_0 \geq 0$, $u_0 \in C_c(X)$, and let $u(x,t) = P_t u_0(x)$.  Since $C_c(X) \subseteq \ell^2(X,m)$, it follows that $u(\cdot, t) \in \ell^2(X,m)$ for all $t\geq0$.  Now, assume that there exists a sequence $x_n \to \infty$ such that $u(x_n,t) \not \to 0$ for some $t$.  By passing to a subsequence, we may assume that $u(x_n,t)> K>0$ for all $n$.  It then follows that
\[ \aV{u(\cdot, t)}^2 \geq \sum_n u(x_n, t)^2 m(x_n) > K^2 \sum_n m(x_n) \]
yielding a contradiction.  It follows that $u(\cdot, t) \in C_0(X)$ for all $t \geq0$.}

\begin{remark}
\begin{itemize}
\item[(i)]  The condition on the measure above also implies the essential self-adjointness of the Laplacian, see \cite[Theorem 6]{KL} and \cite[Corollary 9.2]{HKLW}.
\item[(ii)]  For spherically symmetric graphs defined below, the condition that $m(X)=\infty$ suffices to show that the graph is Feller, see Corollary \ref{model corollary}.  However, this does not hold for general graphs, see Remark \ref{stability remark}.
\end{itemize}
\end{remark}

We now prove criteria for the (failure of the) Feller property using a modified curvature-type quantity. In order to do this we will twist the curvatures by a spherically symmetric function. These results seem to have no analogues in the manifold setting.  They will be used in some of the examples found in Subsection \ref{sub_ex_stab}.  Note that the conditions here have to hold only for \emph{some} $x_0$ in contrast to Theorem \ref{general_thm} where the condition has to hold for all $x_0$.

\begin{definition}
A function $f$ is said to be \emph{spherically symmetric with respect to} $x_0$ if the values of $f$ only depend on the distance to the vertex $x_0$.  In this case, we will write $f(r)$ for $f(x)$ when $x \in S_r(x_0)$.
\end{definition}

We first show that if a modified curvature is not decaying too strongly in all directions from some $x_0$, then the graph is Feller.

\begin{theorem}\label{general_thm2}
Let $f>0$ be spherically symmetric with respect to $x_0 \in X$ such that $\oh{f}(r) = f(r)-f(r+1) > 0$ for all $r > R$ for some $R$ and $f(r) \to 0$ as $r \to \infty$.  If 
\[ \ka_+(x) - \ka_-(x) \left(\frac{\oh{f}(r-1)}{\oh{f}(r)} \right) \geq \lambda \frac{f(r)}{\oh{f}(r)} \]
for all $x\in S_r(x_0)$, $r > R$, and $\lambda<0$,
then $G$ is Feller.
\end{theorem}

\begin{proof}
It is easy to see that the condition above is equivalent to $\ow{\Delta} f(x) \geq \lambda f(x)$ for all $x \not \in B_R(x_0)$ since, for $f$ spherically symmetric with respect to $x_0$,
\begin{align*}
 \ow{\Lp} f(x) &=\ka_+(x)(f(r)-f(r+1)) + \ka_-(x)(f(r)-f(r-1))\\
  &= \ka_+(x) \oh{f}(r) - \ka_-(x) \oh{f}(r-1). 
\end{align*}
Letting $\om = B_R(x_0)$ and rescaling $f$ so that $f\equiv 1$ on $S_R(x_0)$, now gives the result by Theorem \ref{subharmonic}.
\end{proof}

\begin{corollary}\label{growth_corollary}
If for all $x \in S_r(x_0)$, $r>R$ for some $R \geq 1$,  $x_0 \in X$ and $\lm<0$
\[ \ka_+(x) - \ka_-(x) \left(\frac{r+1}{r-1}\right) \geq \lambda(r+1) \] 
 then $G$ is Feller.  
\end{corollary}
\begin{proof}
This follows directly by letting $f(x) = \frac{1}{\rho(x)}$ where $\rho(x)=d(x,x_0)$ in Theorem \ref{general_thm2}.
\end{proof}

In the stochastically complete case, one can state a similar criterion for the failure of the Feller property.  Recall that stochastic completeness means that heat is conserved in the graph at all times, equivalently, that bounded solutions of the heat equation are uniquely determined by initial data.  Some conditions for stochastic completeness in the weighted graph setting are given in \cite{D2, KL, Hua, Hua4, GHM, KLW, Fol, Hua3}, amongst other works.

\begin{theorem}\label{general_thm3}
Let $G$ be stochastically complete and let $f>0$ be bounded and spherically symmetric with respect to $x_0 \in X$  such that $\oh{f}(r) = f(r)-f(r+1) > 0$ for all $r > R$ for some $R$ and $f(r) \not \to 0$ as $r \to \infty$.  If
\[ \ka_+(x) - \ka_-(x) \left(\frac{\oh{f}(r-1)}{\oh{f}(r)} \right) \leq \lambda \frac{f(r)}{\oh{f}(r)} \]
 for all $x\in S_r(x_0)$, $r > R$ for some $R$ and $\lambda<0$, then $G$ is not Feller.
\end{theorem}
\begin{proof}
The proof follows analogously to the proof of Theorem \ref{general_thm2} by using Theorem \ref{superharmonic}.
\end{proof}

\begin{corollary}\label{decay_corollary}
If $G$ is stochastically complete and for all $x \in S_r(x_0)$, $r> R$, for some $R\geq1$, $x_0 \in X$ and $\lm<0$
\[ \ka_+(x) - \ka_-(x) \left(\frac{r+1}{r-1}\right) \leq \lambda(r+1)^2 \] 
then $G$ is not Feller.  
\end{corollary}
\begin{proof}
This follows by letting $f(x) = \frac{1}{\rho(x)} +1$ where $\rho(x) = d(x,x_0)$ in Theorem \ref{general_thm3}.
\end{proof}


\subsection{The spherically symmetric case}
We now give a full characterization of the Feller property in the weakly spherically symmetric case.  The result in the case of manifolds is due to Azencott \cite{Az}; we follow the proof given by Pigola and Setti \cite{PS}.

\begin{definition}
A weighted graph $G$ is called \emph{weakly spherically symmteric} or \emph{model} if there exists a vertex $x_0$ such that the curvatures $\ka_\pm$ are spherically symmetric with respect to $x_0$.
\end{definition}

In this case, we will denote the vertex $x_0$ by $o$ and call it the \emph{root} of the model.  We will also write
\[ \ka_\pm(x) = \vka_\pm(r) \]
for all $x \in S_r(o)$ where $\vka_\pm:\N_0 \to \R$ in preparation for our comparison theorems below.  We will, in general, suppress the dependence on $o$ and simply write $S_r$ for $S_r(o)$ and $B_r$ for $B_r(o) = \cup_{i=0}^r S_i$.  Furthermore, we let $B_r^c$ denote the complement of $B_r$ in $X$.

\eat{
The property of being weakly spherically symmetric is equivalent to several other conditions such as both the Laplacian and heat semigroup commuting with an averaging operator, see \cite[Theorem 1]{KLW} for more details.  In particular, this gives that $p_t(o, \cdot)$ is a spherically symmetric function with respect to the root.
}

We also define
\[ \bbr{r}= \sum_{x \in S_r} \sum_{y \in S_{r+1}} b(x,y) = \vka_+(r) m(S_r). \]
Furthermore, note that,
\begin{equation}\label{curvatures}
\vka_+(r) m(S_r) = \vka_-(r+1)m(S_{r+1})
\end{equation}
for all $r$.  This will be used in several places below.

\begin{theorem}\label{model theorem}
Let $G$ be model.  Then $G$ is Feller if and only if either
\begin{itemize}
\item[(1)] $\displaystyle{\sum_r \frac{1}{\bbr{r}} < \infty}$ or 
\item[(2)] $\displaystyle{\sum_r \frac{1}{\bbr{r}} = \infty \textup{ and }\sum_r \frac{m(B_r^c)}{\bbr{r}} = \infty.}$
\end{itemize}
\end{theorem}

\begin{remark}
Let us compare the result above to those for recurrence and stochastic completeness.  In \cite[Proposition 6.1]{Hua2}, see also \cite[Theorem 5.9]{Woe}, it is shown that the transcience of a model graph is equivalent to $\sum_r \frac{1}{\bbr{r}} < \infty.$
Furthermore, in \cite[Theorem 5]{KLW}, it is shown that a model graph is stochastically incomplete if and only if $\sum_r \frac{m(B_r)}{\bbr{r}} < \infty. $
Hence, all stochastically incomplete and all transient model graphs are Feller which is not true for general graphs as we discuss later, see Remark \ref{stability remark}.

\end{remark}

As an immediate corollary of Theorem \ref{model theorem} above, we get that model graphs of infinite measure are always Feller.
\begin{corollary}\label{model corollary}
If $G$ is a model graph with $m(X) = \infty$, then $G$ is Feller. 
\end{corollary}
Thus, in the case of model graphs, the main interest is in graphs of finite measure.  For more information on many properties of such graphs see \cite{GHKLW}.  Also, note that the statement of the corollary is not true for general graphs, see Remark \ref{stability remark}.

\bigskip

We start the proof of Theorem \ref{model theorem} by showing the key properties of the minimal, positive solution to the exterior $\lambda$-harmonic boundary problem (\ref{minimal solution}) in the model case when $\om = \{ o \}$.

\begin{lemma}\label{lemma ssh}
Let $G$ be model.  Let $h$ be the minimal solution of (\ref{minimal solution}) with $\om = \{ o \}$ and $\lambda < 0$.  Then
\begin{itemize}
\item[(i)]  $h$ is a spherically symmetric function with respect to $o$.
\item[(ii)] $\h(r+1) < \h(r)$ for all $r$.
\item[(iii)] If $f(r) = \bbr{r} (\h(r) - \h(r+1)),$
then $f$ is a positive, decreasing function.  In particular, $\lim_{r \to \infty} f(r)$ exists.
\item[(iv)] 
If $\sum_r \frac{1}{\bbr{r}} = \infty$, then $\lim_{r \to \infty} f(r) = 0$ and
\[ - \lm \lim_{s \to \infty} h(s) m(B_r^c) \leq f(r) \leq -\lm h(r+1) m(B_r^c). \]
\end{itemize}
\end{lemma}

\begin{proof}
Exhaust $G$ by $\om_n = B_n$ and let $h_n$ solve
\[
\left \{
\begin{array}{ll}
\Lp h_n = \lambda h_n & \textup{on } \inte{n}\setminus \{ o \} \\
h_n(o) =1 & \\
h_n = 0 & \textup{on } \bd{n}
\end{array}
\right.
\]
so that $0< h_n<1$ on $\inte{n} \setminus \{ o \}$ and $h_n \to h$ by applying Proposition \ref{e maximum principle}.

The proof of (i) follows by averaging the functions $h_n$ over spheres to obtain a spherically symmetric solution.  That is, let
\[ g_n(r) = \frac{1}{m(S_r)} \sum_{x \in S_r} h_n(x) m(x) \]
for $0 \leq r \leq n$.  A calculation using (\ref{curvatures}) gives that $\Lp g_n(r) = \lm g_n(r)$ and, by a maximum principle such as Proposition \ref{e maximum principle}, it follows that $g_n(r) = h_n(x)$ for all $x \in S_r$ so that $h_n$ is spherically symmetric with respect to $o$.  Taking the limit then proves (i).

To prove (ii), we first claim that $\h_n(r+1) < \h_n(r)$ for $0 \leq r \leq n-1$.
This is clear for $\h_n(n-1) > \h_n(n) =0$.  Then
\begin{align*}
 \Lp \h_n(n-1) &= \vka_+(n-1) \h_n(n-1) + \vka_-(n-1) (\h_n(n-1) - \h_n(n-2)) \\
 & = \lambda \h_n(n-1) < 0
\end{align*}
gives that $\h_n(n-1) < \h_n(n-2)$.  Iterating proves the claim.

Therefore, by taking the limit, we get that
\[ \h(r+1) \leq \h(r) \textup{ for all } r. \]
If there exists an $r$ such that $\h(r+1) = \h(r)$, then
\[ \ow{\Lp} \h(r+1) = \vka_+(r+1) (\h(r+1) - \h(r+2)) = \lambda \h(r+1) < 0 \]
yielding that $\h(r+2) > \h(r+1)$ contradicting the above.  Therefore, $\h(r+1) < \h(r)$.

To prove (iii), first note that, by property (ii), we have that $f(r)>0$.
Multiplying
\[ \ow{\Lp} \h(r) = \vka_+(r) ( \h(r) - \h(r+1)) + \vka_-(r) (\h(r) - \h(r-1)) = \lambda \h(r) \]
by $m(S_r)$ yields
\[ \bbr{r} (\h(r)-\h(r+1)) + \bbr{r-1}(\h(r)-\h(r-1)) = \lambda  \h(r) m(S_r) \]
so that
\begin{equation}\label{recursion}
 f(r)-f(r-1) = \lambda \h(r) m(S_r)<0
\end{equation}
and, therefore, $f$ is decreasing.  

To prove (iv), assume that $\sum_r \frac{1}{\bbr{r}} =\infty$ and $\lim_{r \to \infty} f(r) = \alpha>0$.  Then $f(r)\geq \alpha>0$ for all $r$ which implies that
\[ \h(r) - \h(r+1) \geq \frac{\alpha}{\bbr{r}}. \]
Hence, 
\[ \h(0) - \lim_{s \to \infty} \h(s) = \sum_{r=0}^\infty (\h(r) - \h(r+1)) \geq \sum_{r=0}^\infty \frac{\alpha}{\bbr{r}} \]
contradicting $\sum_r \frac{1}{\bbr{r}} =\infty$.  Therefore, $\lim_{r \to \infty} f(r)=0$ in this case.

Now, summing (\ref{recursion}) gives
\[ \sum_{i=r+1}^\infty (f(i)-f(i-1)) = - f(r) = \lambda \sum_{i=r+1}^\infty \h(i) m(S_i). \]
Using the monotonicity of $h$ then implies that
\[ - \lm \lim_{s \to \infty} h(s) m(B_r^c) \leq f(r) \leq -\lm h(r+1) m(B_r^c). \]

\end{proof}

\begin{remark} \label{uniqueness remark}
Let us point out how stochastic completeness is necessary in Corollary \ref{uniqueness corollary} which states that the minimal positive solution the $\lm$-harmonic exterior elliptic boundary problem is unique.  Namely, by the above, $\h$ is a decreasing function.  On the other hand, if a model graph is stochastically incomplete, then there exists a positive, spherically symmetric, bounded function $v$ which satisfies $\ow{\Lp} v = \lm v$ and $v(o)=1$ \cite[Lemma 5.4]{KLW}.  Furthermore, it can be easily seen by induction that $v$ is increasing \cite[Lemma 4.3]{KLW}.  In particular, $v \not \equiv h$. 
\end{remark}

\begin{proof}[Proof of Theorem \ref{model theorem}]
Recall that $G$ is Feller if and only if $\lim_{s \to \infty}h(s)=0$.

\noindent $\Longrightarrow$:  \quad  Assume that condition (1) holds, that is, $\sum_r \frac{1}{\bbr{r}} < \infty$, and let $\h$ be the minimal solution of (\ref{minimal solution}) with $\om = \{o\}$.  

Let $g(r) = \sum_{i=r}^\infty \frac{1}{\bbr{i}}$
and note that, using (\ref{curvatures}), $\ow{\Lp} g(r) = 0$ for $r>0$.  Letting $v(r) = \frac{g(r)}{g(0)}$ it follows that, for any $\lambda<0$, $v$ satisfies
\[ \ow{\Lp} v(r)   \geq \lambda v(r) \textup{ for } r>0 \textup{ and } v(0)=1. \]
Furthermore, $v(r) \to 0$ as $r \to \infty$.  Therefore, $G$ is Feller by Theorem \ref{subharmonic}.

Assume now that condition (2) holds, that is, $\sum_r \frac{1}{\bbr{r}}=\infty$ and $\sum_r \frac{m(B_r^c)}{\bbr{r}} = \infty$.  
Letting  $f(r) = \bbr{r}(\h(r)-\h(r+1))$
it follows by Lemma \ref{lemma ssh} (iv) above that $f(r) \to 0$ as $r \to \infty$ since $\sum_r \frac{1}{\bbr{r}} = \infty$ and
\begin{equation}{\label{inequality}}
f(r) \geq -\lambda \lim_{s \to \infty} \h(s) m(B_r^c).
\end{equation}

If $m(B_r^c)=\infty$, then $\lim_{s \to \infty} \h(s) = 0$ by (\ref{inequality}) and hence $G$ is Feller.

If $m(B_r^c) < \infty$, then (\ref{inequality}) gives that
\[ \h(r) - \h(r+1) \geq -\lambda \lim_{s \to \infty} \h(s) \frac{m(B_r^c)}{\bbr{r}}. \]
Summing this implies
\begin{align*} 
\sum_{i=r+1}^\infty (\h(i) - \h(i+1)) &= \h(r+1) - \lim_{s \to \infty} \h(s) \geq - \lambda \lim_{s \to \infty} \h(s) \sum_{i=r+1}^\infty \frac{m(B_i^c)}{\bbr{i}} 
\end{align*}
contradicting $\sum_r \frac{m(B_r^c)}{\bbr{r}} = \infty$ if $\lim_{s \to \infty} \h(s) \not = 0$.  Therefore, $\lim_{s \to \infty} \h(s)=0$ and $G$ is Feller.

\bigskip

\noindent $\Longleftarrow$:  \quad  Assume that $G$ is Feller and that $\sum_r \frac{1}{\bbr{r}} = \infty.$
By Lemma \ref{lemma ssh} (iv),
\[ f(r)= \bbr{r}(\h(r) - \h(r+1)) \leq -\lambda \h(r+1)m(B_r^c) \]
so that
\begin{equation}\label{inequality2}
 \frac{\h(r)-\h(r+1)}{\h(r+1)} \leq -\lambda \frac{m(B_r^c)}{\bbr{r}}.
\end{equation}
As
\[ \frac{\h(r)}{\h(r+1)} -1 \geq \ln\left(\frac{\h(r)}{\h(r+1)}\right) = \ln(\h(r)) - \ln(\h(r+1)) \]
and
\[ \sum_{r=0}^\infty ( \ln(\h(r)) - \ln(\h(r+1)) ) = -\lim_{r \to \infty} \ln(\h(r)) = \infty, \]
since $\h(r) \to 0$ as $r \to \infty$,
it follows, by (\ref{inequality2}), that  $\sum_{r} \frac{m(B_r^c)}{\bbr{r}} = \infty.$
\end{proof}

\subsection{Comparison theorems}
We prove here theorems comparing a general graph with a model one which give analogues to a result of Pigola and Setti \cite[Theorem 5.9]{PS}.  The comparison results here are the opposite of what would be expected given similar results concerning heat kernels, bottom of the spectra, and stochastic completeness in the case of graphs found in \cite{KLW}.  

\begin{definition}
Let $G$ be a graph and $\ow{G}$ be a model graph.

We say that $G$ has \emph{stronger curvature growth outside of a finite set} than $\ow{G}$ if there exists a vertex $x_0$ in $G$ such that
\[ \ka_+(x) \geq \vka_+(r) \textup{ and } \ka_-(x) \leq \vka_-(r) \]
for all vertices $x \in S_r(x_0)$ in $G$ and all $r \geq R$ for some $R$.  

We say that $G$ has \emph{weaker curvature growth outside of a finite set} than $\ow{G}$ if $G$ contains a vertex $x_0$ so that the opposite inequalities hold.
\end{definition}

\begin{theorem} \label{comparison_theorem}
Let $G$ be a graph and $\ow{G}$ be a model graph.
\begin{itemize}
\item[(1)] If $G$ has stronger curvature growth outside of a finite set than $\ow{G}$ and $\ow{G}$ is Feller, then $G$ is Feller.
\item[(2)] If $G$ has weaker curvature growth outside of a finite set than $\ow{G}$ and $\ow{G}$ is not Feller, then $G$ is not Feller.
\end{itemize}
\end{theorem}

\begin{proof}
Without loss of generality, we can take $R=0$ in the definition of stronger and weaker curvature growth.  Let $\h$ be the minimal positive solution to the $\lambda$-harmonic exterior boundary problem on $\ow{G}$ and recall, by Lemma \ref{lemma ssh}, that $\h$ is spherically symmetric with respect to $o$ and decreasing.  

Define a function on $G$ by letting
\[ v(x) = \h(r) \textup{ for } x \in S_r(x_0). \]
Then, $v(x_0)=1$, $v>0$ and, using the fact that $G$ has stronger curvature growth than $\ow{G}$ in statement (1),
\begin{align*} 
\ow{\Lp} v(x) &= \ka_+(x) (\h(r) - \h(r+1)) + \ka_-(x) (\h(r) - \h(r-1)) \\
&\geq \vka_+(r) (\h(r) - \h(r+1)) + \vka_-(r)(\h(r) - \h(r-1)) = \lm \h(r) = \lm v(x) 
\end{align*}
so that $\ow{\Lp} v(x) \geq \lm v(x)$ on $X \setminus \{x_0\}$.  Since $\ow{G}$ is Feller, it follows that $\h(r) \to 0$ as $r \to \infty$ so that $v \in C_0(X)$ and $G$ is Feller by Theorem \ref{subharmonic}.  This proves (1).

For the proof of (2), note that the assumption that $\ow{G}$ is not Feller implies, by Theorem \ref{model theorem} and \cite[Theorem 5]{KLW}, that $\ow{G}$ is stochastically complete.  The fact that $G$ has weaker curvature growth than $\ow{G}$ now implies that $G$ is stochastically complete as well \cite[Theorem 6]{KLW}. 
Now, the proof of (2) is similar to that above using Theorem \ref{superharmonic}. 
\end{proof}

\subsection{Examples and stability} \label{sub_ex_stab}
In this subsection we give several examples to illustrate the results above.  We also briefly discuss the stability of the Feller property.

\begin{example}\label{model example}
We start with an example to illustrate Theorem \ref{general_thm}, Theorem \ref{model theorem} and Corollary \ref{growth_corollary}.  In particular, we show that the curvature must be bounded with respect to every base point $x_0$ in order to apply Theorem \ref{general_thm}. 

Let $X = \N_0$ with $b(x,y)=1$ if and only if $|x-y|=1$ and 0 otherwise.  Let 
\[ m(r)=\frac{1}{(r+1)^{2+\ep}} \textup{ for } \ep>0. \]  
Then $G$ is model with  $\bbr{r}=1$ so that $\sum_r \frac{1}{\bbr{r}} = \infty$ and $\sum_r \frac{m(B_r^c)}{\bbr{r}}<\infty$.  Hence, by Theorem \ref{model theorem}, $G$ is not Feller.

Letting $x_0 =o = 0$, it follows that $\ow{\ka}_+(r)= (r+1)^{2+\ep}= \ow{\ka}_-(r)$, so that $\ow{\ka}_+(r) - \ow{\ka}_-(r) = 0$ and
\[ \ow{\ka}_+(r) - \ow{\ka}_-(r) \left( \frac{r+1}{r-1} \right) = \frac{-2(r+1)^{2+\ep}}{r-1} \]
so that Corollary \ref{growth_corollary} does not apply.

Finally, letting the basepoint $x_0$ for $\rho(x) = d(x,x_0)$ vary, say $x_0=n$, it follows that 
\[ \ka_+(n) =2(n+1)^{2+\ep} \]
while $\ka_-(n)=0$ and $\ka_+(x) - \ka_-(x) \leq 0$ for all other vertices, so that Theorem \ref{general_thm} does not apply.
\end{example}

\begin{example}\label{anti_example}
We also give an example to show that no assumption of the form $\ka_+(x) - \ka_-(x) \geq K$ for all $x$ and all $x_0$ implies the Feller property.  In fact, we give examples of graphs with arbitrarily large curvature growth in all directions that are not Feller.  This is somewhat surprising in light of Theorem \ref{comparison_theorem} above which states that any graph with larger curvature growth than a Feller model graph is Feller and in light of the manifold case for which all Cartan-Hadamard manifolds are Feller.

The example is a tree where every vertex has three forward neighbors.  We arrange the vertices in terms of generations so that it will be model as follows: the first generation consists of one vertex, call it $x_{0,1,1}$.  The next generation consists of three neighbors of $x_{0,1,1}$, call them $x_{1,1,1}, x_{1,1,2}$ and $x_{1,1,3}$.  In turn, each of these will have three forward neighbors and, in general, the $r^{\tiny{\textup{th}}}$ generation will have $3^r$ vertices, labeled $x_{r,i,j}$ where $i=1,2, \ldots, 3^{r-1}$ indicates which member of the previous generation that the vertex is connected to and $j=1,2,3$.  Now, we specify the edge weights by letting, $b$ be symmetric with 
\[ b(x_{r,i,j}, x_{r+1,k,l}) = \frac{2(r+1)}{3^{r+1}} \textup{ if and only if } k=j \]
and 0 otherwise.  Therefore, with $o = x_{0,1,1}$, the graph is model with $\bbr{r} = 2(r+1)$ so that $\sum_r \frac{1}{\bbr{r}} = \infty$.

Now, choose the measure so that it is spherically symmetric with respect to $o$ and so that $\sum_r \frac{m(B_r^c)}{\bbr{r}} < \infty$
for example, let $m(r) = ({3^r(r+1)^{1+\ep}})^{-1}$.  Thus, by Theorem \ref{model theorem}, this graph is not Feller. 

On the other hand, 
\[ \ow{\ka}_+(r)- \ow{\ka}_-(r) = \frac{2}{m(r)3^r} \]
for $r>0$ so that, by choosing $m(r)$ appropriately, the graph can have arbitrarily large curvature growth.  Likewise, it is not difficult to see that choosing the basepoint to be in any generation of the tree also produces curvatures which are always positive and can be made arbitrarily large by the choice of $m$.  Thus, curvature growth of any magnitude does not ensure that a graph will be Feller.

Also, note that
\[ \ow{\ka}_+(r)- \ow{\ka}_-(r)\left(\frac{r+1}{r-1} \right) = \frac{-2}{m(r)3^r}\left(\frac{r+1}{r-1}\right) \]
which allows us to apply Corollaries \ref{growth_corollary} and \ref{decay_corollary} to this example.

\end{example}

\begin{remark} \label{stability remark}
We end this subsection with a remark concerning the stability of the Feller property.  Namely, since the elliptic characterization of the Feller property, Theorem \ref{Az theorem}, takes place outside of a finite set, it follows that ``gluing'' a graph which is Feller with one that is not Feller together at a single vertex produces a non-Feller graph.  On the other hand, it is known that such an operation does not affect either transcience \cite{Woe2} or stochastic incompleteness \cite{Hua, KL}.  Therefore, as has already been mentioned, although stochastic incompleteness and transcience do imply the Feller property in the model case, this is certainly not true for general graphs.  Furthermore, ``gluing'' a model graph of infinite measure, which is Feller by Corollary \ref{model corollary}, with one of finite measure which is not Feller, will produce a graph of infinite measure which is not Feller.  Thus, Corollary \ref{model corollary} does not hold for general graphs as well.

Finally, let us point how a non-Feller graph can become Feller by ``gluing'' infinitely many vertices.  Start with the non-Feller graph in Example \ref{model example} with $\ep=1$.  That is, let $X = \N_0$ with $b(x,y)=1$ if and only if $|x-y|=1$ and 0 otherwise and let  $m(r)=(r+1)^{-3}.$  Now, to every vertex $r>1$ in this graph, attach a single additional vertex, call it $x_r$ and extend $b$ so that it is symmetric and $b(r,x_r)>0$.  It then follows that, with respect to the base point $x_0=0$, for $r>1$
\[ \ka_+(r) - \ka_-(r) \left(\frac{r+1}{r-1}\right) = (r+1)^{3} \left( \frac{-2 +b(r,x_r)(r-1)}{r-1} \right) \]
and, since $\ka_+(x_r)=0$ and $d(x_r,0) = r+1$,
\[ \ka_+(x_r) - \ka_-(x_r)\left( \frac{r+2}{r} \right) = -\frac{b(r,x_r) (r+2)}{m(x_r)r} .\]

Therefore, if we let
$ b(r,x_r) = \frac{c}{r-1}$
with $c \geq 2$ and  $m(x_r)\geq(r (r-1))^{-1}$, it then follows that the resulting graph is Feller by Corollary \ref{growth_corollary}.  On the other hand, letting $0 < c<2$ and $m(x_r)\leq (r(r-1)(r+2))^{-1}$ gives that the resulting graph is non-Feller by Corollary \ref{decay_corollary}.  Note that in order to apply Corollary \ref{decay_corollary} we need to check that the resulting graph is stochastically complete.  However, since the original graph is stochastically complete, the stochastic completeness of the resulting graph follows easily by an argument such as given in the proof of Theorem 4.4 in \cite{Woj11}.
\end{remark}

\end{document}